\documentclass{amsart}
\usepackage{amssymb}
\usepackage{amsmath}
\usepackage{amscd}
\usepackage[left=3cm, right=2.5cm, top=2cm, bottom=2cm]{geometry}
\usepackage{hyperref}

\newtheorem{prop}{Proposition}[section]

\newtheorem{lemma}[prop]{Lemma}
\newtheorem{theorem}[prop]{Theorem}

\theoremstyle{definition}

\DeclareFontEncoding{OT2}{}{} 

\begin{document}
\title{Some New Congruences for $l$-Regular Partitions Modulo $l$}
\author{S. Abinash}
\author{T. Kathiravan}
\author{K. Srilakshmi}
\keywords{Congruences, $l$-Regular Partitions, Theta function identities}
\subjclass[2010]{11P83, 05A17}

\address{Indian Institute of Science Education and Research Thiruvananthapuram, Maruthamala P.O., Vithura, Thiruvananthapuram-695551, Kerala, India.}
\email{sarmaabinash15@iisertvm.ac.in}
\address{The Institute of Mathematical Sciences, IV Cross Road, CIT Campus, Taramani, Chennai-600113, Tamil Nadu, India.}
\email{kkathiravan98@gmail.com}
\address{Indian Institute of Science Education and Research Thiruvananthapuram, Maruthamala P.O., Vithura, Thiruvananthapuram-695551, Kerala, India.}
\email{srilakshmi@iisertvm.ac.in}

\begin{abstract}
A partition of $n$ is $l$-regular if none of its parts is divisible by $l$. Let $b_l(n)$ denote the number of $l$-regular partitions of $n$. In this paper, using theta function identities due to Ramanujan, we establish some new infinite families of congruences for $b_l(n)$ modulo $l$, where $l=13,17,23$.
\end{abstract}

\maketitle
\section{Introduction}
We first recall a few terminologies and notations. Let $n$ be a positive integer. A partition of $n$ is a non-increasing sequence of positive integers whose sum is $n$, and the members of the sequence are called parts. For a positive integer $l\geq2$, a partition of $n$ is said to be $l$-regular if none of its parts is divisible by $l$. Let $b_l(n)$ denote the number of $l$-regular partitions of $n$. By convention, we assume that $b_l(0)=1$. The generating function for $b_l(n)$ is given by
\begin{equation}\label{eq:1.1}
\sum_{n=0}^\infty b_l(n)q^n=\frac{f_l}{f_1},
\end{equation}
where $|q|<1$ and for any positive integer $k$, $f_k$ is defined by
\begin{equation*}
f_k:=\prod_{i=1}^\infty\Big(1-q^{ki}\Big).
\end{equation*}
It can be easily verified that for any prime number $l$,
\begin{equation}\label{eq:1.2}
f_l\equiv f_1^l\pmod l.
\end{equation}

In recent years, numerous congruences satisfied by $l$-regular partitions for different values of $l$ has been established. For instance, Lovejoy and Penniston \cite{19} established criteria for $3$-divisibility of $b_3(n)$. Dandurand and Penniston \cite{7} employed the theory of complex multiplication to give a precise description of those $n$ such that $l|b_l(n)$ for $l\in\{5,7,11\}$. Cui and Gu \cite{8} derived infinite families of congruences modulo $2$ for $b_l(n)$ where $l\in\{2,4,5,8,13,16\}$. Webb \cite{9} established an infinite family of congruences modulo $3$ for $b_{13}(n)$. Furcy and Penniston \cite{10} obtained families of congruences modulo $3$ for other values of $l$ which are congruent to $1$ modulo $3$. Xia and Yao \cite{11} found some infinite families of congruences for $b_9(n)$ modulo $2$, and Cui and Gu \cite{12} derived congruences for $b_9(n)$ modulo $3$. Xia \cite{1} established infinite families of congruences for $b_l(n)$ modulo $l$ where $l\in\{13,17,19\}$, for example, for any $n\geq0$, $k\geq0$, 
\begin{equation*}
b_{13}\Bigg(5^{12k}n+\frac{1}{2}\Big(5^{12k}-1\Big)\Bigg)\equiv b_{13}(n)\pmod{13}.
\end{equation*}
For a more complete literature review, see \cite{1}. For more related works, see \cite{13,14,15,16,17,18,20,21,2,26,25,27,22,23,24,28}.

In this paper, we establish infinite families of arithmetic properties and Ramanujan-type congruences for $b_l(n)$ modulo $l$, where $l\in\{13,17,23\}$. In fact, no such congruence has been established for $23$-regular partitions till now, to the best of the authors' knowledge. The following are the main results.
\begin{theorem}\label{thm4}
For any $n\geq0$, $k\geq0$,
\begin{equation}\label{eq:1.3}
b_{13}\Bigg(7^{4k}n+\frac{1}{2}\Big(7^{4k}-1\Big)\Bigg)\equiv2^kb_{13}(n)\pmod{13},
\end{equation}
and for $r\in\{0,1,\cdots,6\}\setminus\{3\}$,
\begin{equation}
b_{13}\Bigg(7^{4k+4}n+\frac{1}{2}\Big(7^{4k+3}(2r+1)-1\Big)\Bigg)\equiv0\pmod{13}.
\end{equation}
\end{theorem}
\begin{theorem}\label{thm1}
For any $n\geq0$, $k\geq0$,
\begin{equation}\label{eq:1.5}
b_{17}\Bigg(5^{4k}n+\frac{2}{3}\Big(5^{4k}-1\Big)\Bigg)\equiv2^k b_{17}(n)\pmod{17},
\end{equation}
and for $r\in\{0,1,2,4\}$,
\begin{equation}
b_{17}\Bigg(5^{4k+4}n+\frac{1}{3}\Big(5^{4k+3}(3r+1)-2\Big)\Bigg)\equiv0\pmod{17}.
\end{equation}
\end{theorem}
\begin{theorem} \label{thm2}
For any $n\geq0$, $k\geq0$,
\begin{equation}\label{eq:1.6}
b_{23}\Bigg(5^{24k}n+\frac{11}{12}\Big(5^{24k}-1\Big)\Bigg)\equiv a(12k)b_{23}(n)\pmod{23},
\end{equation}
where
\begin{equation}\label{eq:1.9}
a(k)=\bigg(\frac{1}{2}-\frac{\sqrt{15}}{15}\bigg)\big(2+\sqrt{15}\big)^k+\bigg(\frac{1}{2}+\frac{\sqrt{15}}{15}\bigg)\big(2-\sqrt{15}\big)^k,
\end{equation}
and for $r\in\{0,1,2,3\}$,
\begin{equation}
b_{23}\Bigg(5^{24k+24}n+\frac{1}{12}\Big(5^{24k+23}(12r+7)-11\Big)\Bigg)\equiv0\pmod{23}.
\end{equation}
\end{theorem}

\section{Preliminaries}
In this section we shall state two preliminary lemmas.
\begin{lemma}
For
\begin{equation*}
R(q):=\prod_{n=1}^\infty\frac{\big(1-q^{5n-4}\big)\big(1-q^{5n-1}\big)}{\big(1-q^{5n-3}\big)\big(1-q^{5n-2}\big)},
\end{equation*}
we have\\
A. (\cite{6} and \cite{5})
\begin{equation} \label{eq:2.1}
f_1=f_{25}\Bigg(\frac{1}{R(q^5)}-q-q^2R\big(q^5\big)\Bigg).
\end{equation}
B. (\cite{3,4})
\begin{equation} \label{eq:2.3}
\frac{1}{R(q)^5}-11q-q^2R(q)^5=\frac{f_1^6}{f_5^6}.
\end{equation}
C.
\begin{equation} \label{eq:2.2}
\frac{1}{f_1}=\frac{f_{25}^5}{f_5^6}\Bigg(\frac{1}{R(q^5)^4}+\frac{q}{R(q^5)^3}+\frac{2q^2}{R(q^5)^2}+\frac{3q^3}{R(q^5)}+5q^4-3q^5R\big(q^5\big)+2q^6R\big(q^5\big)^2-q^7R\big(q^5\big)^3+q^8R\big(q^5\big)^4\Bigg).
\end{equation}
\end{lemma}
\noindent
Note that the identity in (\ref{eq:2.2}) can be easily established by first substituting $q$ by $q^5$ in (\ref{eq:2.3}) and then using (\ref{eq:2.1}).
\begin{lemma}
For
\begin{equation*}
A=A(q):=\frac{f(-q^3,-q^4)}{f(-q^2)},\quad B=B(q):=\frac{f(-q^2,-q^5)}{f(-q^2)},\quad C=C(q):=\frac{f(-q,-q^6)}{f(-q^2)},
\end{equation*}
where $f(a,b):=\sum_{n=-\infty}^\infty a^{n(n+1)/2}b^{n(n-1)/2}$ for $|ab|<1$ and $f(-q):=f(-q,-q^2)$, we have\\
A. (\cite[p. 303, Entry 17(v)]{29})
\begin{equation} \label{eq:2.4}
f_1=f_{49}\Bigg(\frac{B\big(q^7\big)}{C(q^7)}-q\frac{A\big(q^7\big)}{B(q^7)}-q^2+q^5\frac{C\big(q^7\big)}{A(q^7)}\Bigg).
\end{equation}
B. (\cite[p. 174, Entry 31]{30} and \cite{31})
\begin{align}
\frac{B^5}{AC^4}-\frac{A^5}{B^4C}-\frac{q^3C^5}{A^4B}&=3q,\label{eq:2.5}\\
\frac{AB^2}{C^3}+\frac{qA^2C}{B^3}-\frac{q^2BC^2}{A^3}&=\frac{f_1^4}{f_7^4}+8q,\label{eq:2.6}\\
\frac{A^3}{BC^2}-\frac{qB^3}{A^2C}-\frac{q^2C^3}{AB^2}&=\frac{f_1^4}{f_7^4}+5q,\label{eq:2.7}\\
\frac{B^7}{C^7}-\frac{qA^7}{B^7}+\frac{q^5C^7}{A^7}&=\frac{f_1^8}{f_7^8}+14q\frac{f_1^4}{f_7^4}+57q^2.\label{eq:2.8}
\end{align}
\end{lemma}

\section{Congruences for 13-Regular Partitions}
\begin{proof}[\textbf{Proof of Theorem \ref{thm4}}]
For $k=0$, (\ref{eq:1.3}) is trivial.

Taking $l=13$ in (\ref{eq:1.1}), we get
\begin{equation*}
\sum_{n=0}^\infty b_{13}(n)q^n=\frac{f_{13}}{f_1}.
\end{equation*}
Now by (\ref{eq:1.2}),
\begin{equation} \label{eq:6.1}
\sum_{n=0}^\infty b_{13}(n)q^n\equiv f_1^{12}\pmod{13}.
\end{equation}
Replacing $f_1$ using (\ref{eq:2.4}), we get
\begin{equation*}
\sum_{n=0}^\infty b_{13}(n)q^n\equiv f_{49}^{12}\Bigg(\frac{B\big(q^7\big)}{C(q^7)}-q\frac{A\big(q^7\big)}{B(q^7)}-q^2+q^5\frac{C\big(q^7\big)}{A(q^7)}\Bigg)^{12}\pmod{13}.
\end{equation*}
If we extract the terms involving $q^{7n+3}$, divide by $q^3$, and substitute $q$ by $q^{1/7}$, then we get
\begin{equation*}
\begin{aligned}[b]
\sum_{n=0}^\infty b_{13}(7n+3)q^n&\equiv f_{7}^{12}\Bigg(\bigg(\frac{A^3B^6}{C^9}+\frac{2AB^9}{C^{10}}\bigg)\\
&\quad+q\bigg(\frac{A^{10}}{B^8C^2}+\frac{9A^8}{B^5C^3}+\frac{2A^6}{B^2C^4}+\frac{3A^4B}{C^5}+\frac{B^{10}}{A^2C^8}+\frac{8A^2B^4}{C^6}+\frac{2B^7}{C^7}\bigg)\\
&\quad+q^2\bigg(-\frac{2A^9C}{B^{10}}-\frac{2A^7}{B^7}-\frac{4A^5}{B^4C}+\frac{9B^8}{A^3C^5}+\frac{9A^3}{BC^2}+\frac{4B^5}{AC^4}+\frac{12AB^2}{C^3}\bigg)\\
&\quad+q^3\bigg(\frac{A^6C^3}{B^9}+\frac{8A^4C^2}{B^6}+\frac{2B^6}{A^4C^2}+\frac{12A^2C}{B^3}-\frac{9B^3}{A^2C}+9\bigg)\\
&\quad+q^4\bigg(\frac{3B^4C}{A^5}-\frac{12BC^2}{A^3}-\frac{3AC^4}{B^5}-\frac{9C^3}{AB^2}\bigg)+q^5\bigg(\frac{8B^2C^4}{A^6}-\frac{4C^5}{A^4B}+\frac{2C^6}{A^2B^4}\bigg)\\
&\quad+q^6\bigg(-\frac{B^3C^6}{A^9}+\frac{2C^7}{A^7}-\frac{9C^8}{A^5B^3}\bigg)+q^7\bigg(\frac{C^{10}}{A^8B^2}-\frac{2BC^9}{A^{10}}\bigg)\Bigg)\pmod{13}.
\end{aligned}
\end{equation*}
After rearranging, we get
\begin{align*}
\sum_{n=0}^\infty b_{13}(7n+3)q^n&\equiv f_7^{12}\Bigg(2\bigg(\frac{AB^9}{C^{10}}-\frac{q^2A^9C}{B^{10}}-\frac{q^7BC^9}{A^{10}}\bigg)+\bigg(\frac{A^3B^6}{C^9}+\frac{q^3A^6C^3}{B^9}-\frac{q^6B^3C^6}{A^9}\bigg)\\
&\quad+q\bigg(\frac{A^{10}}{B^8C^2}+\frac{B^{10}}{A^2C^8}+\frac{q^6C^{10}}{A^8B^2}\bigg)+2q\bigg(\frac{B^7}{C^7}-\frac{qA^7}{B^7}+\frac{q^5C^7}{A^7}\bigg)\\
&\quad+8q\bigg(\frac{A^2B^4}{C^6}+\frac{q^2A^4C^2}{B^6}+\frac{q^4B^2C^4}{A^6}\bigg)+3q\bigg(\frac{A^4B}{C^5}+\frac{q^3B^4C}{A^5}-\frac{q^3AC^4}{B^5}\bigg)\\
&\quad+2q\bigg(\frac{A^6}{B^2C^4}+\frac{q^2B^6}{A^4C^2}+\frac{q^4C^6}{A^2B^4}\bigg)+9q\bigg(\frac{A^8}{B^5C^3}+\frac{qB^8}{A^3C^5}-\frac{q^5C^8}{A^5B^3}\bigg)\\
&\quad+4q^2\bigg(\frac{B^5}{AC^4}-\frac{A^5}{B^4C}-\frac{q^3C^5}{A^4B}\bigg)+12q^2\bigg(\frac{AB^2}{C^3}+\frac{qA^2C}{B^3}-\frac{q^2BC^2}{A^3}\bigg)\\
&\quad+9q^2\bigg(\frac{A^3}{BC^2}-\frac{qB^3}{A^2C}-\frac{q^2C^3}{AB^2}\bigg)+9q^3\Bigg)\pmod{13}.
\end{align*}
We rewrite it in the following manner.
\begin{align*}
&\sum_{n=0}^\infty b_{13}(7n+3)q^n\\
&\equiv f_7^{12}\Bigg(2\bigg(\Big(\frac{B^7}{C^7}-\frac{qA^7}{B^7}+\frac{q^5C^7}{A^7}\Big)\Big(\frac{AB^2}{C^3}+\frac{qA^2C}{B^3}-\frac{q^2BC^2}{A^3}\Big)\\
&\quad\quad-q\Big(\frac{A^3}{BC^2}-\frac{qB^3}{A^2C}-\frac{q^2C^3}{AB^2}\Big)\Big(\Big(\frac{B^5}{AC^4}-\frac{A^5}{B^4C}-\frac{q^3C^5}{A^4B}\Big)-q\Big)\bigg)\\
&\quad+\bigg(\Big(\frac{AB^2}{C^3}+\frac{qA^2C}{B^3}-\frac{q^2BC^2}{A^3}\Big)^3-3q\Big(\frac{A^3}{BC^2}-\frac{qB^3}{A^2C}-\frac{q^2C^3}{AB^2}\Big)\Big(\frac{AB^2}{C^3}+\frac{qA^2C}{B^3}-\frac{q^2BC^2}{A^3}\Big)-3q^3\bigg)\\
&\quad+q\bigg(\Big(\frac{B^5}{AC^4}-\frac{A^5}{B^4C}-\frac{q^3C^5}{A^4B}\Big)^2+2q\Big(\frac{B^5}{AC^4}-\frac{A^5}{B^4C}-\frac{q^3C^5}{A^4B}\Big)\\
&\quad\quad+2\Big(\frac{AB^2}{C^3}+\frac{qA^2C}{B^3}-\frac{q^2BC^2}{A^3}\Big)\Big(\frac{A^3}{BC^2}-\frac{qB^3}{A^2C}-\frac{q^2C^3}{AB^2}\Big)+6q^2\bigg)\\
&\quad+2q\bigg(\frac{B^7}{C^7}-\frac{qA^7}{B^7}+\frac{q^5C^7}{A^7}\bigg)+8q\bigg(\Big(\frac{AB^2}{C^3}+\frac{qA^2C}{B^3}-\frac{q^2BC^2}{A^3}\Big)^2-2q\Big(\frac{A^3}{BC^2}-\frac{qB^3}{A^2C}-\frac{q^2C^3}{AB^2}\Big)\bigg)\\
&\quad+3q\bigg(q\Big(\frac{B^5}{AC^4}-\frac{A^5}{B^4C}-\frac{q^3C^5}{A^4B}\Big)+\Big(\frac{AB^2}{C^3}+\frac{qA^2C}{B^3}-\frac{q^2BC^2}{A^3}\Big)\Big(\frac{A^3}{BC^2}-\frac{qB^3}{A^2C}-\frac{q^2C^3}{AB^2}\Big)+3q^2\bigg)\\
&\quad+2q\bigg(\Big(\frac{A^3}{BC^2}-\frac{qB^3}{A^2C}-\frac{q^2C^3}{AB^2}\Big)^2+2q\Big(\frac{AB^2}{C^3}+\frac{qA^2C}{B^3}-\frac{q^2BC^2}{A^3}\Big)\bigg)\\
&\quad+9q\bigg(\Big(\frac{AB^2}{C^3}+\frac{qA^2C}{B^3}-\frac{q^2BC^2}{A^3}\Big)^2-\Big(\frac{A^3}{BC^2}-\frac{qB^3}{A^2C}-\frac{q^2C^3}{AB^2}\Big) \Big(\Big(\frac{B^5}{AC^4}-\frac{A^5}{B^4C}-\frac{q^3C^5}{A^4B}\Big)+q\Big)\bigg)\\
&\quad+4q^2\bigg(\frac{B^5}{AC^4}-\frac{A^5}{B^4C}-\frac{q^3C^5}{A^4B}\bigg)+12q^2\bigg(\frac{AB^2}{C^3}+\frac{qA^2C}{B^3}-\frac{q^2BC^2}{A^3}\bigg)\\
&\quad+9q^2\bigg(\frac{A^3}{BC^2}-\frac{qB^3}{A^2C}-\frac{q^2C^3}{AB^2}\bigg)+9q^3\Bigg)\pmod{13}.
\end{align*}
By (\ref{eq:2.5}), (\ref{eq:2.6}), (\ref{eq:2.7}), and (\ref{eq:2.8}), we get
\begin{equation*}
\begin{aligned}[b]
\sum_{n=0}^\infty b_{13}(7n+3)q^n&\equiv f_7^{12}\Bigg(2\bigg(\Big(\frac{f_1^8}{f_7^8}+q\frac{f_1^4}{f_7^4}+5q^2\Big)\Big(\frac{f_1^4}{f_7^4}+8q\Big)-q\Big(\frac{f_1^4}{f_7^4}+5q\Big)\Big(\Big(3q\Big)-q\Big)\bigg)\\
&\quad+\bigg(\Big(\frac{f_1^4}{f_7^4}+8q\Big)^3-3q\Big(\frac{f_1^4}{f_7^4}+5q\Big)\Big(\frac{f_1^4}{f_7^4}+8q\Big)-3q^3\bigg)\\
&\quad+q\bigg(\Big(3q\Big)^2+2q\Big(3q\Big)+2\Big(\frac{f_1^4}{f_7^4}+8q\Big)\Big(\frac{f_1^4}{f_7^4}+5q\Big)+6q^2\bigg)+2q\bigg(\frac{f_1^8}{f_7^8}+q\frac{f_1^4}{f_7^4}+5q^2\bigg)\\
&\quad+8q\bigg(\Big(\frac{f_1^4}{f_7^4}+8q\Big)^2-2q\Big(\frac{f_1^4}{f_7^4}+5q\Big)\bigg)+3q\bigg(q\Big(3q\Big)+\Big(\frac{f_1^4}{f_7^4}+8q\Big)\Big(\frac{f_1^4}{f_7^4}+5q\Big)+3q^2\bigg)\\
&\quad+2q\bigg(\Big(\frac{f_1^4}{f_7^4}+5q\Big)^2+2q\Big(\frac{f_1^4}{f_7^4}+8q\Big)\bigg)+9q\bigg(\Big(\frac{f_1^4}{f_7^4}+8q\Big)^2-\Big(\frac{f_1^4}{f_7^4}+5q\Big)\Big(\Big(3q\Big)+q\Big)\bigg)\\
&\quad+4q^2\bigg(3q\bigg)+12q^2\bigg(\frac{f_1^4}{f_7^4}+8q\bigg)+9q^2\bigg(\frac{f_1^4}{f_7^4}+5q\bigg)+9q^3\Bigg)\pmod{13},
\end{aligned}
\end{equation*}
which, after rearranging the similar terms, yields
\begin{equation*}
\sum_{n=0}^\infty b_{13}(7n+3)q^n\equiv3f_1^{12}+2q^3f_7^{12}\pmod{13}.
\end{equation*}

In view of (\ref{eq:6.1}), we get
\begin{equation}\label{eq:6.2}
\sum_{n=0}^\infty b_{13}(7n+3)q^n\equiv3\sum_{n=0}^\infty b_{13}(n)q^n+2\sum_{n=0}^\infty b_{13}(n)q^{7n+3}\pmod{13}.
\end{equation}
If we extract the terms involving $q^{7n+3}$, divide by $q^3$, and substitute $q$ by $q^{1/7}$, then we get
\begin{equation*}
\sum_{n=0}^\infty b_{13}\big(7^2n+24\big)q^n\equiv3\sum_{n=0}^\infty b_{13}(7n+3)q^n+2\sum_{n=0}^\infty b_{13}(n)q^n\pmod{13}.
\end{equation*}
Now by (\ref{eq:6.2}),
\begin{align*}
\sum_{n=0}^\infty b_{13}\big(7^2n+24\big)q^n\equiv11\sum_{n=0}^\infty b_{13}(n)q^n+6\sum_{n=0}^\infty b_{13}(n)q^{7n+3}\pmod{13}.
\end{align*}
Repeating the same process as above once more, we get
\begin{equation}\label{eq:6.3}
\sum_{n=0}^\infty b_{13}\big(7^3n+171\big)q^n\equiv2\sum_{n=0}^\infty b_{13}(n)q^{7n+3}\pmod{13}.
\end{equation}

Comparing the coefficients of the terms of the form $q^{7n+3}$, we can conclude that for any $n\geq0$,
\begin{align*}
b_{13}\big(7^4n+1200\big)\equiv2b_{13}(n)\pmod{13},
\end{align*}
or,
\begin{align*}
b_{13}\Bigg(7^4n+\frac{1}{2}\Big(7^4-1\Big)\Bigg)\equiv2b_{13}(n)\pmod{13}.
\end{align*}
This proves (\ref{eq:1.3}) for $k=1$. Suppose
\begin{align*}
b_{13}\Bigg(7^{4(k-1)}n+\frac{1}{2}\Big(7^{4(k-1)}-1\Big)\Bigg)\equiv2^{k-1}b_{13}(n)\pmod{13}.
\end{align*}
If we substitute $n$ by $7^4n+\frac{1}{2}\Big(7^4-1\Big)$, then we get
\begin{align*}
b_{13}\Bigg(7^{4(k-1)}\bigg(7^4n+\frac{1}{2}\Big(7^4-1\Big)\bigg)+\frac{1}{2}\Big(7^{4(k-1)}-1\Big)\Bigg)\equiv2^{k-1}b_{13}\bigg(7^4n+\frac{1}{2}\Big(7^4-1\Big)\bigg)\pmod{13},
\end{align*}
which yields
\begin{align*}
b_{13}\Bigg(7^{4k}n+\frac{1}{2}\Big(7^{4k}-1\Big)\Bigg)\equiv2^kb_{13}(n)\pmod{13}.
\end{align*}
This completes the proof of (\ref{eq:1.3}) by the principle of mathematical induction.

In (\ref{eq:6.3}), comparing the coefficients of the terms of the form $q^{7n+r}$ where $r\in\{0,1,\cdots,6\}\setminus\{3\}$, we can conclude that for any $n\geq0$,
\begin{align*}
b_{13}\Bigg(7^3(7n+r)+\frac{7^3-1}{2}\Bigg)\equiv0\pmod{13},
\end{align*}
or,
\begin{align}\label{eq:6.4}
b_{13}\Bigg(7^4n+\frac{1}{2}\Big(7^3(2r+1)-1\Big)\Bigg)\equiv0\pmod{13}.
\end{align}
Now by (\ref{eq:1.3}) and (\ref{eq:6.4}), we get
\begin{align*}
b_{13}\Bigg(7^{4k+4}n+\frac{1}{2}\Big(7^{4k+3}(2r+1)-1\Big)\Bigg)\equiv0\pmod{13}.
\end{align*}
This completes the proof of the theorem.
\end{proof}

\section{Congruences for 17-Regular Partitions}
\begin{proof}[\textbf{Proof of Theorem \ref{thm1}}]
For $k=0$, (\ref{eq:1.5}) is trivial.

Taking $l=17$ in (\ref{eq:1.1}), we get
\begin{equation*}
\sum_{n=0}^\infty b_{17}(n)q^n=\frac{f_{17}}{f_1}.
\end{equation*}
Now by (\ref{eq:1.2}),
\begin{equation} \label{eq:3.1}
\sum_{n=0}^\infty b_{17}(n)q^n\equiv f_1^{16}\pmod{17}.
\end{equation}
Replacing $f_1$ using (\ref{eq:2.1}) and then extracting the terms involving $q^{5n+1}$, we get
\begin{align*}
\sum_{n=0}^\infty b_{17}(5n+1)q^{5n+1}&\equiv f_{25}^{16}\Bigg(\frac{q}{R(q^5)^{15}}+\frac{13q^6}{R(q^5)^{10}}+\frac{8q^{11}}{R(q^5)^5}-q^{16}-8q^{21}R\big(q^5\big)^5\\
&\quad\quad+13q^{26}R\big(q^5\big)^{10}-q^{31}R\big(q^5\big)^{15}\Bigg)\pmod{17}.
\end{align*}
If we divide by $q$ and substitute $q$ by $q^{1/5}$, then we get
\begin{equation*}
\begin{aligned}[b]
\sum_{n=0}^\infty b_{17}(5n+1)q^n&\equiv f_5^{16}\Bigg(\bigg(\frac{1}{R(q)^5}-11q-q^2R(q)^5\bigg)^3+12q\bigg(\frac{1}{R(q)^5}-11q-q^2R(q)^5\bigg)^2\\
&\quad+14q^2\bigg(\frac{1}{R(q)^5}-11q-q^2R(q)^5\bigg)+7q^3\Bigg)\pmod{17}.
\end{aligned}
\end{equation*}
By (\ref{eq:2.3}), we get
\begin{equation*}
\sum_{n=0}^\infty b_{17}(5n+1)q^n\equiv f_5^{16}\Bigg(\frac{f_1^{18}}{f_5^{18}}+12q\frac{f_1^{12}}{f_5^{12}}+14q^2\frac{f_1^6}{f_5^6}+7q^3\Bigg)\pmod{17}.
\end{equation*}
In view of (\ref{eq:3.1}), we get
\begin{equation}\label{eq:3.3}
\sum_{n=0}^\infty b_{17}(5n+1)q^n\equiv\frac{f_1^{18}}{f_5^2}+12qf_1^{12}f_5^4+14q^2f_1^6f_5^{10}+7\sum_{n=0}^\infty b_{17}(n)q^{5n+3}\pmod{17}.
\end{equation}

Replacing $f_1$ using (\ref{eq:2.1}), we get
\begin{equation*}
\begin{aligned}[b]
\sum_{n=0}^\infty b_{17}(5n+1)q^n&\equiv\frac{f_{25}^{18}}{f_5^2}\Bigg(\frac{1}{R(q^5)}-q-q^2R\big(q^5\big)\Bigg)^{18}+12qf_5^4f_{25}^{12}\Bigg(\frac{1}{R(q^5)}-q-q^2R\big(q^5\big)\Bigg)^{12}\\
&\quad+14q^2f_5^{10}f_{25}^6\Bigg(\frac{1}{R(q^5)}-q-q^2R\big(q^5\big)\Bigg)^6+7\sum_{n=0}^\infty b_{17}(n)q^{5n+3}\pmod{17}.
\end{aligned}
\end{equation*}
If we extract the terms involving $q^{5n+3}$, divide by $q^3$, and substitute $q$ by $q^{1/5}$, then by (\ref{eq:2.3}) we get
\begin{equation*}
\sum_{n=0}^\infty b_{17}(5^2n+16)q^n\equiv\frac{f_5^{18}}{f_1^2}\Bigg(q^3\Bigg)+12f_1^4f_5^{12}\Bigg(3\frac{f_1^{12}}{f_5^{12}}+3q^2\Bigg)+14f_1^{10}f_5^6\Bigg(11\frac{f_1^6}{f_5^6}+9q\Bigg)+7\sum_{n=0}^\infty b_{17}(n)q^n\pmod{17},
\end{equation*}
which, after rearranging the similar terms, yields
\begin{equation*}
\sum_{n=0}^\infty b_{17}\big(5^2n+16\big)q^n\equiv7qf_1^{10}f_5^6+2q^2f_1^4f_5^{12}+q^3\frac{f_5^{18}}{f_1^2}+10\sum_{n=0}^\infty b_{17}(n)q^n\pmod{17}.
\end{equation*}

Replacing $f_1$ and $1/f_1$ using (\ref{eq:2.1}) and (\ref{eq:2.2}), we get
\begin{equation*}
\begin{aligned}[b]
\sum_{n=0}^\infty b_{17}\big(5^2n+16\big)q^n&\equiv7qf_5^6f_{25}^{10}\Bigg(\frac{1}{R(q^5)}-q-q^2R\big(q^5\big)\Bigg)^{10}+2q^2f_5^{12}f_{25}^4\Bigg(\frac{1}{R(q^5)}-q-q^2R\big(q^5\big)\Bigg)^4\\
&+q^3f_5^6f_{25}^{10}\Bigg(\frac{1}{R(q^5)^4}+\frac{q}{R(q^5)^3}+\frac{2q^2}{R(q^5)^2}+\frac{3q^3}{R(q^5)}+5q^4-3q^5R\big(q^5\big)+2q^6R\big(q^5\big)^2\\
&\quad\quad-q^7R\big(q^5\big)^3+q^8R\big(q^5\big)^4\Bigg)^2\\
&+10\sum_{n=0}^\infty b_{17}(n)q^n\pmod{17}.
\end{aligned}
\end{equation*}
If we extract the terms involving $q^{5n+1}$, divide by $q$, and substitute $q$ by $q^{1/5}$, then by (\ref{eq:2.3}) we get
\begin{equation*}
\begin{aligned}[b]
\sum_{n=0}^\infty b_{17}\big(5^3n+41\big)q^n&\equiv7f_1^6f_5^{10}\Bigg(\frac{f_1^{12}}{f_5^{12}}+12q\frac{f_1^6}{f_5^6}+q^2\Bigg)+2f_1^{12}f_5^4\Bigg(12q\Bigg)+f_1^6f_5^{10}\Bigg(10q\frac{f_1^6}{f_5^6}+6q^2\Bigg)\\
&\quad+10\sum_{n=0}^\infty b_{17}(5n+1)q^n\pmod{17},
\end{aligned}
\end{equation*}
which, after rearranging the similar terms, yields
\begin{equation*}
\begin{aligned}[b]
\sum_{n=0}^\infty b_{17}\big(5^3n+41\big)q^n&\equiv7\frac{f_1^{18}}{f_5^2}+16qf_1^{12}f_5^4+13q^2f_1^6f_5^{10}+10\sum_{n=0}^\infty b_{17}(5n+1)q^n\pmod{17}.
\end{aligned}
\end{equation*}
Now by (\ref{eq:3.3}), we get
\begin{equation}\label{eq:3.4}
\sum_{n=0}^\infty b_{17}\big(5^3n+41\big)q^n\equiv2\sum_{n=0}^\infty b_{17}(n)q^{5n+3}\pmod{17}.
\end{equation}

Comparing the coefficients of the terms of the form $q^{5n+3}$, we can conclude that for any $n\geq0$,
\begin{equation*}
b_{17}\big(5^4n+416\big)\equiv2b_{17}(n)\pmod{17},
\end{equation*}
or,
\begin{equation*}
b_{17}\Bigg(5^4n+\frac{2}{3}\Big(5^4-1\Big)\Bigg)\equiv2b_{17}(n)\pmod{17}.
\end{equation*}
This proves (\ref{eq:1.5}) for $k=1$. Suppose
\begin{equation*}
b_{17}\Bigg(5^{4(k-1)}n+\frac{2}{3}\Big(5^{4(k-1)}-1\Big)\Bigg)\equiv2^{k-1}b_{17}(n)\pmod{17}.
\end{equation*}
If we substitute $n$ by $5^4n+\frac{2}{3}\Big(5^4-1\Big)$, then we get
\begin{equation*}
b_{17}\Bigg(5^{4(k-1)}\bigg(5^4n+\frac{2}{3}\Big(5^4-1\Big)\bigg)+\frac{2}{3}\Big(5^{4(k-1)}-1\Big)\Bigg)\equiv2^{k-1}b_{17}\bigg(5^4n+\frac{2}{3}\Big(5^4-1\Big)\bigg)\pmod{17},
\end{equation*}
which yields
\begin{equation*}
b_{17}\Bigg(5^{4k}n+\frac{2}{3}\Big(5^{4k}-1\Big)\Bigg)\equiv2^kb_{17}(n)\pmod{17}.
\end{equation*}
This completes the proof of (\ref{eq:1.5}) by the principle of mathematical induction.

In (\ref{eq:3.4}), comparing the coefficients of the terms of the form $q^{5n+r}$ where $r\in\{0,1,2,4\}$, we can conclude that for any $n\geq0$,
\begin{equation*}
b_{17}\Bigg(5^3(5n+r)+\frac{5^3-2}{3}\Bigg)\equiv0\pmod{17},
\end{equation*}
or,
\begin{equation}\label{eq:3.5}
b_{17}\Bigg(5^4n+\frac{1}{3}\Big(5^3(3r+1)-2\Big)\Bigg)\equiv0\pmod{17}.
\end{equation}
Now by (\ref{eq:1.5}) and (\ref{eq:3.5}), we get
\begin{equation*}
b_{17}\Bigg(5^{4k+4}n+\frac{1}{3}\Big(5^{4k+3}(3r+1)-2\Big)\Bigg)\equiv0\pmod{17}.
\end{equation*}
This completes the proof of the theorem.
\end{proof}

\section{Congruences for 23-Regular Partitions}
We first establish the following lemma.
\begin{lemma}
For any $n\geq0$, $k\geq0$,
\begin{equation}\label{eq:4.1}
b_{23}\Bigg(5^{2k}n+\frac{11}{12}\Big(5^{2k}-1\Big)\Bigg)\equiv a'(k)b_{23}\big(5^2n+22\big)+a(k)b_{23}(n)\pmod{23},
\end{equation}
where $a(k)$ is defined by (\ref{eq:1.9}) and
\begin{equation*}
a'(k)=\frac{\sqrt{15}}{30}\big(2+\sqrt{15}\big)^k-\frac{\sqrt{15}}{30}\big(2-\sqrt{15}\big)^k.
\end{equation*}
\end{lemma}
\begin{proof}
For $k=0,1$, (\ref{eq:4.1}) can be easily verified.

Taking $l=23$ in (\ref{eq:1.1}), we get
\begin{equation*}
\sum_{n=0}^\infty b_{23}(n)q^n=\frac{f_{23}}{f_1}.
\end{equation*}
Now by (\ref{eq:1.2}),
\begin{equation}\label{eq:4.2}
\sum_{n=0}^\infty b_{23}(n)q^n\equiv f_1^{22}\pmod{23}.
\end{equation}
Replacing $f_1$ using (\ref{eq:2.1}) and then extracting the terms involving $q^{5n+2}$, we get
\begin{equation*}
\begin{aligned}[b]
\sum_{n=0}^\infty b_{23}(5n+2)q^{5n+2}&\equiv f_{25}^{22}\Bigg(\frac{2q^2}{R(q^5)^{20}}+\frac{21q^7}{R(q^5)^{15}}+\frac{3q^{12}}{R(q^5)^{10}}+\frac{8q^{17}}{R(q^5)^5}-q^{22}-8q^{27}R\big(q^5\big)^5\\
&\quad+3q^{32}R\big(q^5\big)^{10}-21q^{37}R\big(q^5\big)^{15}+2q^{42}R\big(q^5\big)^{20}\Bigg)\pmod{23}.
\end{aligned}
\end{equation*}
If we divide by $q^2$ and substitute $q$ by $q^{1/5}$, then we get
\begin{equation*}
\begin{aligned}[b]
\sum_{n=0}^\infty b_{23}(5n+2)q^n&\equiv f_5^{22}\Bigg(2\bigg(\frac{1}{R(q)^5}-11q-q^2R(q)^5\bigg)^4+17q\bigg(\frac{1}{R(q)^5}-11q-q^2R(q)^5\bigg)^3\\
&\quad+17q^2\bigg(\frac{1}{R(q)^5}-11q-q^2R(q)^5\bigg)^2+14q^4\Bigg)\pmod{23}.
\end{aligned}
\end{equation*}
By (\ref{eq:2.3}), we get
\begin{equation*}
\sum_{n=0}^\infty b_{23}(5n+2)q^n\equiv f_5^{22}\Bigg(2\frac{f_1^{24}}{f_5^{24}}+17q\frac{f_1^{18}}{f_5^{18}}+17q^2\frac{f_1^{12}}{f_5^{12}}+14q^4\Bigg)\pmod{23}.
\end{equation*}
In view of (\ref{eq:4.2}), we get
\begin{equation}\label{eq:4.3}
\sum_{n=0}^\infty b_{23}(5n+2)q^n\equiv2\frac{f_1^{24}}{f_5^2}+17qf_1^{18}f_5^4+17q^2f_1^{12}f_5^{10}+14\sum_{n=0}^\infty b_{23}(n)q^{5n+4}\pmod{23}.
\end{equation}

Replacing $f_1$ using (\ref{eq:2.1}), we get
\begin{equation*}
\begin{aligned}[b]
\sum_{n=0}^\infty b_{23}(5n+2)q^n&\equiv2\frac{f_{25}^{24}}{f_5^2}\Bigg(\frac{1}{R(q^5)}-q-q^2R\big(q^5\big)\Bigg)^{24}+17qf_5^4f_{25}^{18}\Bigg(\frac{1}{R(q^5)}-q-q^2R\big(q^5\big)\Bigg)^{18}\\
&\quad+17q^2f_5^{10}f_{25}^{12}\Bigg(\frac{1}{R(q^5)}-q-q^2R\big(q^5\big)\Bigg)^{12}+14\sum_{n=0}^\infty b_{23}(n)q^{5n+4}\pmod{23}.
\end{aligned}
\end{equation*}
If we extract the terms involving $q^{5n+4}$, divide by $q^4$, and substitute $q$ by $q^{1/5}$, then by (\ref{eq:2.3}) we get
\begin{align*}
\sum_{n=0}^\infty b_{23}(5^2n+22)q^n&\equiv2\frac{f_5^{24}}{f_1^2}\Bigg(q^4\Bigg)+17f_1^4f_5^{18}\Bigg(19\frac{f_1^{18}}{f_5^{18}}+7q^3\Bigg)+17f_1^{10}f_5^{12}\Bigg(8\frac{f_1^{12}}{f_5^{12}}+3q^2\Bigg)\\
&\quad+14\sum_{n=0}^\infty b_{23}(n)q^n\pmod{23},
\end{align*}
which, after rearranging the similar terms, yields
\begin{equation}\label{eq:4.4}
\sum_{n=0}^\infty b_{23}\big(5^2n+22\big)q^n\equiv5q^2f_1^{10}f_5^{12}+4q^3f_1^4f_5^{18}+2q^4\frac{f_5^{24}}{f_1^2}+13\sum_{n=0}^\infty b_{23}(n)q^n\pmod{23}.
\end{equation}

Replacing $f_1$ and $1/f_1$ using (\ref{eq:2.1}) and (\ref{eq:2.2}), we get
\begin{equation*}
\begin{aligned}[b]
\sum_{n=0}^\infty b_{23}\big(5^2n+22\big)q^n&\equiv5q^2f_5^{12}f_{25}^{10}\Bigg(\frac{1}{R(q^5)}-q-q^2R\big(q^5\big)\Bigg)^{10}+4q^3f_5^{18}f_{25}^4\Bigg(\frac{1}{R(q^5)}-q-q^2R\big(q^5\big)\Bigg)^4\\
&\quad+2q^4f_5^{12}f_{25}^{10}\Bigg(\frac{1}{R(q^5)^4}+\frac{q}{R(q^5)^3}+\frac{2q^2}{R(q^5)^2}+\frac{3q^3}{R(q^5)}+5q^4-3q^5R\big(q^5\big)+2q^6R\big(q^5\big)^2\\
&\quad\quad-q^7R\big(q^5\big)^3+q^8R\big(q^5\big)^4\Bigg)^2\\
&\quad+13\sum_{n=0}^\infty b_{23}(n)q^n\pmod{23}.
\end{aligned}
\end{equation*}
If we extract the terms involving $q^{5n+2}$, divide by $q^2$, and substitute $q$ by $q^{1/5}$, then by (\ref{eq:2.3}) we get
\begin{equation*}
\begin{aligned}[b]
\sum_{n=0}^\infty b_{23}\big(5^3n+72\big)q^n&\equiv5f_1^{12}f_5^{10}\Bigg(\frac{f_1^{12}}{f_5^{12}}+20q\frac{f_1^6}{f_5^6}+16q^2\Bigg)+4f_1^{18}f_5^4\Bigg(18q\Bigg)+2f_1^{12}f_5^{10}\Bigg(10q\frac{f_1^6}{f_5^6}+10q^2\Bigg)\\
&\quad+13\sum_{n=0}^\infty b_{23}(5n+2)q^n\pmod{23},
\end{aligned}
\end{equation*}
which, after rearranging the similar terms, yields
\begin{equation*}
\sum_{n=0}^\infty b_{23}\big(5^3n+72\big)q^n\equiv5\frac{f_1^{24}}{f_5^2}+8qf_1^{18}f_5^4+8q^2f_1^{12}f_5^{10}+13\sum_{n=0}^\infty b_{23}(5n+2)q^n\pmod{23}.
\end{equation*}
Now by (\ref{eq:4.3}), we get
\begin{equation*}
\sum_{n=0}^\infty b_{23}\big(5^3n+72\big)q^n\equiv8\frac{f_1^{24}}{f_5^2}+22qf_1^{18}f_5^4+22q^2f_1^{12}f_5^{10}+21\sum_{n=0}^\infty b_{23}(n)q^{5n+4}\pmod{23}.
\end{equation*}

Replacing $f_1$ using (\ref{eq:2.1}), we get
\begin{equation*}
\begin{aligned}[b]
\sum_{n=0}^\infty b_{23}(5^3n+72)q^n&\equiv8\frac{f_{25}^{24}}{f_5^2}\Bigg(\frac{1}{R(q^5)}-q-q^2R\big(q^5\big)\Bigg)^{24}+22qf_5^4f_{25}^{18}\Bigg(\frac{1}{R(q^5)}-q-q^2R\big(q^5\big)\Bigg)^{18}\\
&\quad+22q^2f_5^{10}f_{25}^{12}\Bigg(\frac{1}{R(q^5)}-q-q^2R\big(q^5\big)\Bigg)^{12}+21\sum_{n=0}^\infty b_{23}(n)q^{5n+4}\pmod{23}.
\end{aligned}
\end{equation*}
If we extract the terms involving $q^{5n+4}$, divide by $q^4$, and substitute $q$ by $q^{1/5}$, then by (\ref{eq:2.3}) we get
\begin{align*}
\sum_{n=0}^\infty b_{23}(5^4n+572)q^n&\equiv8\frac{f_5^{24}}{f_1^2}\Bigg(q^4\Bigg)+22f_1^4f_5^{18}\Bigg(19\frac{f_1^{18}}{f_5^{18}}+7q^3\Bigg)+22f_1^{10}f_5^{12}\Bigg(8\frac{f_1^{12}}{f_5^{12}}+3q^2\Bigg)\\
&\quad+21\sum_{n=0}^\infty b_{23}(n)q^n\pmod{23},
\end{align*}
which, after rearranging the similar terms, yields
\begin{equation*}
\sum_{n=0}^\infty b_{23}\big(5^4n+572\big)q^n\equiv20q^2f_1^{10}f_5^{12}+16q^3f_1^4f_5^{18}+8q^4\frac{f_5^{24}}{f_1^2}+17\sum_{n=0}^\infty b_{23}(n)q^n\pmod{23}.
\end{equation*}
Now by (\ref{eq:4.4}), we get
\begin{equation*}
\sum_{n=0}^\infty b_{23}\big(5^4n+572\big)q^n\equiv4\sum_{n=0}^\infty b_{23}\big(5^2n+22\big)q^n+11\sum_{n=0}^\infty b_{23}(n)q^n\pmod{23}.
\end{equation*}
Comparing the coefficients of $q^n$, we can conclude that for any $n\geq0$,
\begin{equation}\label{eq:4.5}
b_{23}\big(5^4n+572\big)\equiv4b_{23}\big(5^2n+22\big)+11b_{23}(n)\pmod{23}.
\end{equation}
This verifies the lemma for $k=2$.

We assume that
\begin{equation}\label{eq:4.6}
b_{23}\Bigg(5^{2(k-1)}n+\frac{11}{12}\Big(5^{2(k-1)}-1\Big)\Bigg)\equiv a'(k-1)b_{23}\big(5^2n+22\big)+a(k-1)b_{23}(n)\pmod{23},
\end{equation}
and
\begin{equation}\label{eq:4.7}
b_{23}\Bigg(5^{2k}n+\frac{11}{12}\Big(5^{2k}-1\Big)\Bigg)\equiv a'(k)b_{23}\big(5^2n+22\big)+a(k)b_{23}(n)\pmod{23}.
\end{equation}
If we substitute $n$ by $5^{2(k-1)}n+\frac{11}{12}\Big(5^{2(k-1)}-1\Big)$ in (\ref{eq:4.5}), then we get
\begin{equation}\label{eq:4.8}
\begin{aligned}
b_{23}\Bigg(5^{2(k+1)}n+\frac{11}{12}\Big(5^{2(k+1)}-1\Big)\Bigg)&\equiv4b_{23}\Bigg(5^{2k}n+\frac{11}{12}\Big(5^{2k}-1\Big)\Bigg)\\
&\quad+11b_{23}\Bigg(5^{2(k-1)}n+\frac{11}{12}\Big(5^{2(k-1)}-1\Big)\Bigg)\pmod{23}.
\end{aligned}
\end{equation}
It is easy to verify that for any $k\geq1$
\begin{equation}\label{eq:4.9}
a'(k+1)=4a'(k)+11a'(k-1),
\end{equation}
and
\begin{equation}\label{eq:4.10}
a(k+1)=4a(k)+11a(k-1).
\end{equation}
By (\ref{eq:4.6}), (\ref{eq:4.7}), (\ref{eq:4.8}), (\ref{eq:4.9}), and (\ref{eq:4.10}), we can conclude that
\begin{equation*}
b_{23}\Bigg(5^{2(k+1)}n+\frac{11}{12}\Big(5^{2(k+1)}-1\Big)\Bigg)\equiv a'(k+1)b_{23}\big(5^2n+22\big)+a(k+1)b_{23}(n)\pmod{23}.
\end{equation*}
Hence the lemma is proved by the principle of mathematical induction.
\end{proof}
\begin{proof}[\textbf{Proof of Theorem \ref{thm2}}]
In order to prove (\ref{eq:1.6}), it is enough to prove that $a'(12k)\equiv0\pmod{23}$ for any $k\geq0$. Clearly, $a'(0)=0$. By repeatedly applying (\ref{eq:4.9}), we get
\begin{equation*}
\begin{aligned}[b]
a'(12)&\equiv4a'(11)+11a'(10)&&\pmod{23}\\
&\equiv4a'(10)+21a'(9)&&\pmod{23}\\
&\quad\quad\vdots\\
&\equiv18a'(2)+20a'(1)&&\pmod{23}\\
&\equiv14a'(0)&&\pmod{23}.
\end{aligned}
\end{equation*}
Consequently, $a'(12)\equiv0\pmod{23}$. Let's assume $a'(12(k-1))\equiv0\pmod{23}$. Again, by repeatedly applying (\ref{eq:4.9}), we get $a'(12k)\equiv14a'(12(k-1))\pmod{23}$ and hence $a'(12k)\equiv0\pmod{23}$. This completes the proof of (\ref{eq:1.6}) by the principle of mathematical induction.

Putting $k=11$ in (\ref{eq:4.1}), we get
\begin{equation*}
b_{23}\Bigg(5^{22}n+\frac{11}{12}\Big(5^{22}-1\Big)\Bigg)\equiv 18b_{23}\big(5^2n+22\big)+20b_{23}(n)\pmod{23}\quad\forall n\geq0,
\end{equation*}
or
\begin{equation*}
\sum_{n=0}^\infty b_{23}\Bigg(5^{22}n+\frac{11}{12}\Big(5^{22}-1\Big)\Bigg)q^n\equiv 18\sum_{n=0}^\infty b_{23}\big(5^2n+22\big)q^n+20\sum_{n=0}^\infty b_{23}(n)q^n\pmod{23}.
\end{equation*}
Now by (\ref{eq:4.4}), we get
\begin{equation*}
\sum_{n=0}^\infty b_{23}\Bigg(5^{22}n+\frac{11}{12}\Big(5^{22}-1\Big)\Bigg)q^n\equiv21q^2f_1^{10}f_5^{12}+3q^3f_1^4f_5^{18}+13q^4\frac{f_5^{24}}{f_1^2}+\sum_{n=0}^\infty b_{23}(n)q^n\pmod{23}.
\end{equation*}
Replacing $f_1$ and $1/f_1$ using (\ref{eq:2.1}) and (\ref{eq:2.2}), we get
\begin{equation*}
\begin{aligned}[b]
\sum_{n=0}^\infty b_{23}\Bigg(5^{22}n+&\frac{11}{12}\Big(5^{22}-1\Big)\Bigg)q^n\\
&\equiv21q^2f_5^{12}f_{25}^{10}\Bigg(\frac{1}{R(q^5)}-q-q^2R\big(q^5\big)\Bigg)^{10}+3q^3f_5^{18}f_{25}^4\Bigg(\frac{1}{R(q^5)}-q-q^2R\big(q^5\big)\Bigg)^4\\
&+13q^4f_5^{12}f_{25}^{10}\Bigg(\frac{1}{R(q^5)^4}+\frac{q}{R(q^5)^3}+\frac{2q^2}{R(q^5)^2}+\frac{3q^3}{R(q^5)}+5q^4-3q^5R\big(q^5\big)+2q^6R\big(q^5\big)^2\\
&\quad\quad-q^7R\big(q^5\big)^3+q^8R\big(q^5\big)^4\Bigg)^2\\
&+\sum_{n=0}^\infty b_{23}(n)q^n\pmod{23}.
\end{aligned}
\end{equation*}
If we extract the terms involving $q^{5n+2}$, divide by $q^2$, and substitute $q$ by $q^{1/5}$, then by (\ref{eq:2.3}) we get
\begin{equation*}
\begin{aligned}[b]
\sum_{n=0}^\infty b_{23}\Bigg(5^{23}n+\frac{7\cdot5^{23}-11}{12}\Bigg)q^n&\equiv21f_1^{12}f_5^{10}\Bigg(\frac{f_1^{12}}{f_5^{12}}+20q\frac{f_1^6}{f_5^6}+16q^2\Bigg)+3f_1^{18}f_5^4\Bigg(18q\Bigg)\\
&\quad+13f_1^{12}f_5^{10}\Bigg(10q\frac{f_1^6}{f_5^6}+10q^2\Bigg)+\sum_{n=0}^\infty b_{23}(5n+2)q^n\pmod{23},
\end{aligned}
\end{equation*}
which, after rearranging the similar terms, yields
\begin{equation*}
\sum_{n=0}^\infty b_{23}\Bigg(5^{23}n+\frac{7\cdot5^{23}-11}{12}\Bigg)q^n\equiv21\frac{f_1^{24}}{f_5^2}+6qf_1^{18}f_5^4+6q^2f_1^{12}f_5^{10}+\sum_{n=0}^\infty b_{23}(5n+2)q^n\pmod{23}.
\end{equation*}
Now by (\ref{eq:4.3}), we get
\begin{equation*}
\sum_{n=0}^\infty b_{23}\Bigg(5^{23}n+\frac{7\cdot5^{23}-11}{12}\Bigg)q^n\equiv14\sum_{n=0}^\infty b_{23}(n)q^{5n+4}\pmod{23}.
\end{equation*}

Comparing the coefficients of the terms of the form $q^{5n+r}$ where $r\in\{0,1,2,3\}$, we can conclude that for any $n\geq0$,
\begin{equation*}
b_{23}\Bigg(5^{23}(5n+r)+\frac{7\cdot5^{23}-11}{12}\Bigg)\equiv0\pmod{23},
\end{equation*}
or,
\begin{equation}\label{eq:4.11}
b_{23}\Bigg(5^{24}n+\frac{1}{12}\Big(5^{23}(12r+7)-11\Big)\Bigg)\equiv0\pmod{23}.
\end{equation}
Now by (\ref{eq:1.6}) and (\ref{eq:4.11}), we get
\begin{equation*}
b_{23}\Bigg(5^{24k+24}n+\frac{1}{12}\Big(5^{24k+23}(12r+7)-11\Big)\Bigg)\equiv0\pmod{23}.
\end{equation*}
This completes the proof of the theorem.
\end{proof}

\end{document}